\documentclass[final]{IEEEtran}
\usepackage{amsfonts}
\usepackage[sort]{cite}
\usepackage{amsmath} 
\usepackage{amssymb} 
\usepackage{amsthm} 
\usepackage{graphicx}
\usepackage{epsfig}
\usepackage{epstopdf}
\usepackage{mathtools}
\usepackage{etoolbox}
\usepackage{lipsum}
\usepackage{multicol}
\usepackage{dblfloatfix}

\usepackage{soul}
\usepackage[colorlinks=true]{hyperref}
\usepackage[normal]{subfigure}
\usepackage{verbatim}

\usepackage{algorithm,algorithmic}

\def\IID{\mathop{\mathrm{i.i.d.}}}  
\def\trace{\mathop{\mathrm{trace}}} 
\def\lhs{\mathop{\mathrm{LHS}}}     
\def\rhs{\mathop{\mathrm{RHS}}}     
\def\srd{\mathop{\mathrm{SRD}}}     
\def\rdf{\mathop{\mathrm{RDF}}}     
\def\sdp{\mathop{\mathrm{SDP}}}     
\def\mse{\mathop{\mathrm{MSE}}}     
\def\rv{\mathop{\mathrm{RV}}}       
\def\th{\mathop{\mathrm{th}}}       
\hypersetup{
     colorlinks = true,
     linkcolor = [rgb]{0,0,1},
     anchorcolor = [rgb]{0,0,1},
     citecolor = [rgb]{0.9,0.5,0},
     filecolor = [rgb]{0,0,1},
     pagecolor = [rgb]{1,1,0},
     urlcolor = [rgb]{0.9,0.5,0},
     bookmarks,
        bookmarksopen = true,
        bookmarksnumbered = true,
        breaklinks = true,
        linktocpage,
        pagebackref,
        colorlinks = true,
        linkcolor = [rgb]{0.2,0.6,0.2},
        urlcolor  = [rgb]{0,0,1},
        citecolor = [rgb]{0.9,0.5,0},
        anchorcolor = [rgb]{0.2,0.6,0.2},
        hyperindex = true,
        hyperfigures
}

\newtheorem{theorem}{Theorem}
\newtheorem{lemma}{Lemma}

\newtheorem{remark}{Remark}

\newcommand{\T}{^{\mbox{\tiny T}}}

\newcommand{\be}{\begin{equation}}
\newcommand{\ee}{\end{equation}}
\newcommand{\bea}{\begin{eqnarray}}
\newcommand{\eea}{\end{eqnarray}}
\newcommand{\bes}{\begin{eqnarray*}}
\newcommand{\ees}{\end{eqnarray*}}

\newcommand{\bfi}{\begin{figure}}
\newcommand{\bfit}{\begin{figure}[t]}
\newcommand{\bfib}{\begin{figure}[b]}
\newcommand{\bfih}{\begin{figure}[h]}
\newcommand{\bfip}{\begin{figure}[p]}
\newcommand{\efi}{\end{figure}}

\newcommand{\bi}{\begin{itemize}}
\newcommand{\ei}{\end{itemize}}
\newcommand{\ben}{\begin{enumerate}}
\newcommand{\een}{\end{enumerate}}

\newcommand{\ep}{\end{problem}}

\newcommand{\bx}{{\bf x}}
\newcommand{\by}{{\bf y}}
\newcommand{\bw}{{\bf w}}
\newcommand{\bv}{{\bf v}}
\newcommand{\bq}{{\bf q}}
\newcommand{\bz}{{\bf z}}
\newcommand{\bu}{{\bf u}}

\begin{document}

\sloppy

\title{The Time-Invariant Multidimensional Gaussian Sequential Rate-Distortion Problem Revisited}

\author{{\IEEEauthorblockN{Photios A. Stavrou, Takashi Tanaka, and Sekhar Tatikonda
\thanks{\IEEEauthorblockA{P. A. Stavrou is with the Department of Electronic Systems at Aalborg University, Denmark}. {\it e-mail:} {\tt fos@es.aau.dk}.}
\thanks{\IEEEauthorblockA{T. Tanaka is with the Department of Aerospace Engineering and Engineering Mechanics at the University of Texas at Austin, TX, USA}. {\it e-mail:} {\tt ttanaka@utexas.edu}.}
\thanks{\IEEEauthorblockA{S. Tatikonda is with the Department of Statistics and Data Science, Yale University, New Haven, CT, USA}. {\it e-mail:} {\tt sekhar.tatikonda@yale.edu}.}
}}}


\maketitle

%
%
%
%
\begin{abstract}
We revisit the sequential rate-distortion (SRD) trade-off problem for vector-valued Gauss-Markov sources with mean-squared error distortion constraints. 
We show via a counterexample that the dynamic reverse water-filling algorithm suggested by  \cite[eq. (15)]{tatikonda:2004} is not applicable to this problem, and consequently the closed form expression of the asymptotic SRD function derived in \cite[eq. (17)]{tatikonda:2004} is not correct in general. Nevertheless, we show that the multidimensional Gaussian SRD function is semidefinite representable and thus it is readily computable.
\end{abstract}

\begin{IEEEkeywords}
Sequential rate distortion (SRD) function, semidefinite programming (SRD), multidimensional, Gauss-Markov process, counterexample. 
\end{IEEEkeywords}

%
%
%
%
\section{Introduction}
\label{sec:introduction}

\par The sequential rate-distortion (SRD) trade-off problem, formally introduced by Tatikonda {\it et. al.} in \cite[Section IV]{tatikonda:2004} based on the earlier works of Gorbunov and Pinsker \cite{gorbunov-pinsker1972b,gorbunov-pinsker1972a}, can be viewed as a variant of the classical rate-distortion trade-off problem \cite{berger:1971} in which causality constraints are strictly imposed. 
Tatikonda {\it et. al.} also introduced the concept of SRD function, which is defined similarly to the classical rate-distortion function (RDF) with an additional requirement that the reconstructed random process depends on the source random process only in a causal manner.  

\par In \cite[Section IV]{tatikonda:2004}, the authors also studied the operational interpretations of the SRD function in the analysis of zero-delay communication systems.
In particular, it was shown that the SRD function provides a lower bound to the smallest data-rate achievable by the class of zero-delay source codes satisfying the given distortion constraints. 
This result was further exploited to evaluate fundamental performance limitations of feedback control systems over communication channels. 
Derpich and $\O$stergaard in \cite{derpich:2012} showed that the $\srd$ function is a lower bound to both the \emph{operational causal RDF} and the \emph{operational zero-delay RDF}\footnote{See, e.g., \cite{neuhoff:1982,linder:2014,stavrou:2017,wood:2017} for the definitions of causal and zero-delay $\rdf$ functions.}. Moreover, they showed that this lower bound is achievable
by a zero-delay source coder with lattice quantizers up to a constant space-filling loss.
Additional works on the operational meaning of the $\srd$ function can be found, for instance, in \cite{stavrou2016,kostina:2016}. 

\par These results show that the SRD function plays an important role to characterize the fundamental performance limitation of real-time communication systems and feedback control systems over communication channels. The purpose of this note is to revisit the existing results regarding the computation of the SRD function and correct an error in the literature.

\subsection{Related Literature}\label{subsec:related_literature}

\par Gorbunov and Pinsker \cite{gorbunov-pinsker1972a} characterized the finite-time SRD function for time-varying and stationary vector-valued Gauss-Markov processes with per-letter mean-squared error (MSE) distortion. For scalar-valued Gauss-Markov processes, they gave the expression of the finite-time SRD using the reverse-waterfilling optimization at each time instant. 
Bucy \cite{bucy:1982} considered the sensor-estimator joint design problem for Gauss-Markov processes in which the mean-square estimation error is minimized subject to the data-rate constraint. The optimal solution derived in \cite{bucy:1982} turned out to coincide with the optimal solution to the corresponding SRD problem derived in 
\cite{tatikonda:2004}. This result shed light on the ``sensor-estimator separation principle,'' asserting that an optimal solution to the SRD problem for Gauss-Markov processes can always be realized as a two-stage mechanism comprised of a linear memoryless sensor with Gaussian noise followed by the Kalman filter. Derpich and ${\O}$stergaard  \cite{derpich:2012} derived  bounds of the asymptotic SRD function for stationary, stable scalar-valued Gaussian autoregressive models with per-letter MSE distortion. They have also derived the closed form expression of the asymptotic SRD function of a stationary, stable scalar-valued Gaussian autoregressive model with unit memory. 
To our knowledge, the most general expression of the optimal solution to the SRD problem (with general sources and general distortion criteria) is given by Stavrou {\it et. al.} in \cite[Theorem 1]{stavrou-charalambous-charalambous2016}. Tanaka {\it et al.} \cite{tanaka:2017} studied the multidimensional Gaussian SRD problem subject to the weighted per-letter MSE distortion constraint by revisiting the sensor-estimator separation principle. They showed that the considered SRD problem can be reformulated as a log-determinant maximization problem \cite{vandenberghe1998determinant}, which can be solved by the standard semidefinite programming (SDP) solver.

\subsection{Contributions}\label{sub:contributions}

\par In this technical note, we revisit the SRD framework of \cite[Section IV]{tatikonda:2004} and re-examine some of the fundamental results derived therein for time-invariant multidimensional Gauss-Markov processes subject to a per-letter MSE distortion. 
\par As the first contribution, we prove via a counterexample, that the dynamic reverse-waterfilling algorithm of \cite[p. 14, eq. (15)]{tatikonda:2004} cannot be applied to the considered problem,\footnote{We note that Kostina and Hassibi in \cite{kostina:2016} questioned the correctness of the expressions of \cite[p. 14, eq. (15), eq. (17)]{tatikonda:2004} without, however, giving the precise reasons of their observations.} and consequently the expression  \cite[eq. (17)]{tatikonda:2004} of the asymptotic limit of the SRD function is not correct in general. 
\par As the second contribution, we provide a correct expression of the asymptotic limit of SRD function using a semidefinite representation, based on an earlier result \cite{tanaka:2017}.
This means that the value of the asymptotic limit of SRD function can be computed by
semidefinite programming (SDP).
\par The rest of this technical note is structured as follows. In Section \ref{sec:problem_formulation}, we formulate the finite-time SRD function of time-invariant vector-valued Gauss-Markov processes under per-letter MSE distortion criteria and its per unit time asymptotic limit. In Section \ref{sec:prior_work}, we review some structural results on the considered $\srd$ problem.
Section \ref{sec:main_results} presents the main results of this technical note and in Section \ref{sec:conclusions} we draw conclusions.

\paragraph*{\bf Notation} Let ${\cal X}$ be a complete separable metric space, and ${\cal B}_{{\cal X}}$ be the Borel $\sigma$-algebra on ${\cal X}$. Let the triplet $(\Omega,{\cal F},{\cal P})$ be a probability space and ${\bf x}:(\Omega,{\cal F})\longmapsto({\cal X},{\cal B}_{\cal X})$ be a random variable. We use lower case boldface letters such as ${\bf x}$, to denote random variable while $x\in{\cal X}$ denotes the realization of ${\bf x}$. For a random variable ${\bf x}$,  we denote the probability distribution induced by ${\bf x}$ on $({\cal X}, {\cal B}_{\cal X})$ by ${\bf P}_{X}(dx)\equiv {\bf P}(dx)$. We denote the conditional distribution of ${\bf y}$ given ${\bf x}=x$ by ${\bf P}_{{\bf y}|{\bf x}}(dy|{\bf x}=x)  \equiv {\bf P}(dy|x)$. We denote random vectors ${\bf x}^n=({\bf x}_0,\ldots,x_n)$ and ${\bf x}^{-1}=({\bf x}_{-\infty},\ldots,{\bf x}_{-1})$. We denote by $A \succ 0$ (respectively, $A \succeq 0$) a positive-definite matrix (respectively, positive-semidefinite matrix). We denote by $I_p\in\mathbb{R}^{p\times{p}}$ the p-dimensional identity matrix. For a positive-semidefinite matrix $\Theta$, we write $\|x\|_\Theta\triangleq\sqrt{x\T\Theta{x}}$.

%
%
%
%
\section{Problem Formulation}
\label{sec:problem_formulation}

\par In this section, we recall the definition of the finite time SRD function with per-letter MSE distortion criteria and its per unit time asymptotic limit. Let the distributions of the source random process ${\bf x}$ and the reconstruction random process ${\bf y}$ be given by
\begin{align}
{\bf P}(dx^n)&\triangleq\prod\nolimits_{t=0}^n{\bf P}(dx_t|{x}^{t-1}), \label{eqsource} \\
{\bf P}(dy^n||x^n)&\triangleq\prod\nolimits_{t=0}^n{\bf P}(dy_t|y^{t-1},x^t).  \label{eqreconstruction}
\end{align}
We assume that ${\bf P}(dx_0|{x}^{-1})={\bf P}(dx_0)$ and ${\bf P}(dy_0|y^{-1},x^0)={\bf P}(dy_0|x_0)$. Denote by ${\bf P}(d{x}^n,d{y}^n)\triangleq{\bf P}(dx^n)\otimes{\bf P}(dy^n||x^n)$ the joint distribution, and let ${\bf P}(d{y}_t|{y}^{t-1})$ be the marginal on ${y}_t\in{\cal Y}_t$ induced by the joint distribution ${\bf P}(d{x}^n,d{y}^n)$.
In the general SRD problem, the source distribution \eqref{eqsource} is given, while the reconstruction distribution \eqref{eqreconstruction} is to be synthesized to minimize the mutual information $I({\bf x}^n;{\bf y}^n)$ subject to a certain distortion constraint. Notice that the mutual information under the considered setting admits the following expressions:
\begin{subequations}
\label{mutual:information}
\begin{align}
I({\bf x}^n;{\bf y}^n)&\triangleq\sum\nolimits_{t=0}^n{I}({\bf x}^n;{\bf y}_t|{\bf y}^{t-1}),~t<n\nonumber\\
&\stackrel{(a)}=\sum\nolimits_{t=0}^n{I}({\bf x}^t;{\bf y}_t|{\bf y}^{t-1}),\label{eq.1a}\\
&=\sum\nolimits_{t=0}^n\mathbb{E}\log\left(\frac{{\bf P}(d{\bf y}_t|{\bf y}^{t-1},{\bf x}^t)}{{\bf P}(d{\bf y}_t|{\bf y}^{t-1})}\right),\label{eq.1b}
\end{align}
\end{subequations}
where (a) follows from the condition independence ${\bf P}(dy_t|y^{t-1},x^n)={\bf P}(dy_t|y^{t-1},x^t),~\forall(x^n,y^{t-1})$, and $\mathbb{E}\{\cdot\}$ is the expectation with respect to the joint probability distribution ${\bf P}(d{x}^n,d{y}^n)$.

\subsection{Finite-time Gaussian SRD function}

\par Next, we formally introduce the finite-time SRD function of time-invariant vector-valued Gauss-Markov sources subject to weighted per-letter MSE distortion criteria studied by Tatikonda {\it et al.} in \cite[Section IV]{tatikonda:2004}.
Let ${\bf x}_t$ be a time-invariant $\mathbb{R}^p$-valued Gauss-Markov process
\begin{align}
{\bf x}_{t+1}=A{\bf x}_t+{\bf w}_t,~t=0,\ldots,n,\label{prob_stat:eq.1}
\end{align} 
where $A\in\mathbb{R}^{p\times{p}}$ is a deterministic matrix, ${\bf x}_0\sim{\cal N}(0;\Sigma_{{\bf x}_0})$ is the initial state with $\Sigma_{{\bf x}_0}\succ 0$, and ${\bf w}_t\in\mathbb{R}^p\sim{\cal N}(0;\Sigma_{{\bf w}})$, is a white Gaussian noise process independent of ${\bf x}_0$. The finite-time $\srd$ function is defined by 
\begin{subequations}\label{prob_stat1}
\begin{align}
{R}^{\srd}_{0,n}(D) \triangleq & \inf_{{\bf P}(dy^n||x^n)}\frac{1}{n+1}I({\bf x}^n;{\bf y}^n),\label{prob_stat:eq.2}\\
&\mbox{s.t.}~~\mathbb{E}\|{\bf x}_t-{\bf y}_t\|_{\Theta_t}^2\leq{D},~~\forall{t=0, \ldots , n}\label{prob_stat:eq.3}
\end{align}
\end{subequations}
provided the infimum exists.
For simplicity, we assume $\Theta_t=I_p$ in the sequel. The extension of the results to general $\Theta_t \succeq 0$ is straightforward.

\subsection{Asymptotic Limits}\label{subsec:infinite_horizon}

\par Let ${\bf x}_t$ be the time-invariant $\mathbb{R}^p$-valued Gauss-Markov process of \eqref{prob_stat:eq.1}. The per unit time asymptotic limit of \eqref{prob_stat1} is defined by
\begin{align}
{R}^{\srd}(D) &\triangleq \lim_{n\longrightarrow\infty}{R}^{\srd}_{0,n}(D),\label{infinite:eq.2}
\end{align}
provided the limit exists. 
\begin{remark}
For unstable Gauss-Markov processes (i.e., matrix $A$ in \eqref{prob_stat:eq.1} has eigenvalues with magnitude greater than one), then we must have \cite{tatikonda:2004,nair-evans2004}: 
\begin{align}
{R}^{\srd}(D)\geq\sum_{\lambda_i(A)>1}\log|\lambda_i(A)|,\label{unstable_eig:eq.1}
\end{align}
where $\lambda_i(A)$ denotes the $i^{\th}$ eigenvalue of matrix $A$.
\end{remark}

\par If  we interchange the $\lim$ and $\inf$ in \eqref{infinite:eq.2}, we obtain the following expression:
\begin{align}
\widehat{R}^{\srd}(D) \triangleq &\inf_{{\bf P}(dy^\infty||x^\infty)}\lim_{n\longrightarrow\infty}\frac{1}{n+1}I({\bf x}^n;{\bf y}^n),\label{stationary:eq.3}\\
&\mbox{s.t.}~~\mathbb{E}\|{\bf x}_t-{\bf y}_t\|^2\leq{D}, \; \forall t
\end{align}
where ${\bf P}(dy^\infty||x^\infty)$ denotes the sequence of conditional probability distributions ${\bf P}(dy_t|y^{t-1},x^t)$, $t=0,1,\ldots$. Note that ${R}^{\srd}(D)\leq\widehat{R}^{\srd}(D)$ holds trivially.

%
%
%
%

\section{Prior Work on $\srd$ function for Time-Invariant Gauss-Markov Sources}\label{sec:prior_work}

\par In this section, we provide some structural results derived in \cite{tatikonda:2004} and \cite{tanaka:2017} for the optimization problem \eqref{prob_stat1}. We also summarize explicit expressions of $R^{\srd}(D)$ and $\widehat{R}^{\srd}(D)$ that are available in the literature.

\subsection{Structural results of the optimal solution}\label{subsec:structural_results}

\begin{lemma}
\label{lemma:1}
Let the source ${\bf x}_t\in\mathbb{R}^p$ be the Gauss-Markov process described by \eqref{prob_stat:eq.1}. Then, the  minimizer for \eqref{prob_stat1} can be chosen with the form 
\begin{align}
{\bf P}^*(dy_t|y^{t-1},x^t)\equiv{\bf P}^*(dy_t|y^{t-1},x_t),~t=0,\ldots,n.\label{optimal:solution}
\end{align}
Moreover, for each $t$, \eqref{optimal:solution} is conditionally Gaussian probability distribution that can be realized by a linear equation of the form
\begin{align}
{\bf y}_t=\bar{A}_t{\bf x}_t+\bar{B}_t{\bf y}^{t-1}+{\bf v}_t,~t=0,\ldots,n,\label{linear:expression}
\end{align} 
where $\bar{A}_t\in\mathbb{R}^{p\times{p}}$ and $\bar{B}_t\in\mathbb{R}^{p\times{t}p}$ are matrices, and ${\bf v}_t\sim{N}(0;\Sigma_{{\bf v}_t})$ is a random variable independent of $(\bx_0, \bw^t, \bv^{t-1})$ and $\Sigma_{{\bf v}_t}\succeq 0$ for each $t=0, \ldots , n$.
\end{lemma}
\begin{proof}
The proof is found in \cite[Lemma 4.3]{tatikonda:2004}.
\end{proof}

\par The following two lemmas strengthen Lemma~\ref{lemma:1}.

\begin{lemma}\label{lemma:2}
In Lemma~\ref{lemma:1}, the minimizer process \eqref{linear:expression} can also be written as 
\begin{subequations}
\label{eq:xqy}
\begin{align}
\bq_t &= \bar{A}_t\bx_t+\bv_t, \\
\by_t &=\mathbb{E}(\bx_t|\bq^t),
\end{align}
\end{subequations}
where ${\bf v}_t\sim{N}(0;\Sigma_{{\bf v}_t})$ is independent of $(\bx_0, \bw^t, \bv^{t-1})$ and $\Sigma_{{\bf v}_t}\succeq 0$ for $t=0, \ldots, n$.
Here, the matrices $\bar{A}_t, \Sigma_{\bv_t}$, $t=0,\ldots, n$ are chosen equally to those in \eqref{linear:expression}.
\end{lemma}
\begin{proof}
The derivation is given in \cite{tanaka:2017}. For completeness we include the proof in Appendix \ref{proof:lemma:2}. (Author's comment: In the final version, Appendix \ref{proof:lemma:2} may be omitted.)
\end{proof}

\begin{lemma}\label{lemma:3}
If $R_{0,n}^{{\srd}}(D)<\infty$, then the minimizer process \eqref{linear:expression} can  be written as
\begin{subequations} 
\label{eqpy}
\begin{align}
{\bf p}_t &= E_t\bx_t+\bz_t, \\
\by_t &=\mathbb{E}(\bx_t|{\bf p}^t),
\end{align}
\end{subequations}
where $\bz_t\sim N(0, \Sigma_{\bz_t})$ is independent of $(\bx_0, \bw^t, \bz^{t-1})$ and $\Sigma_{\bz_t}\succ 0$ for $t=0,\ldots, n$. 
\end{lemma}
\begin{proof}
The derivation is given in \cite{tanaka:2017}. For completeness, we include the proof in Appendix \ref{proof:lemma:3}. (Author's comment: In the final version, Appendix \ref{proof:lemma:3} may be omitted.)
\end{proof}

Notice that the result of Lemma~\ref{lemma:3} is stronger than that of Lemma~\ref{lemma:2} in that the covariance matrix $\Sigma_{\bz_t}$ can always be chosen as a strictly positive-definite matrix. Lemma~\ref{lemma:2} and \ref{lemma:3} are also different in that the dimension of ${\bf q}_t$ is always $p$, while the dimension of ${\bf p}_t$ can be smaller than $p$.

\subsection{Expressions of $\srd$ functions}

\par The authors of \cite{tatikonda:2004} obtained the following explicit form of $R^{\srd}(D)$ for time-invariant \emph{scalar-valued} Gauss-Markov processes subject to the MSE distortion\footnote{This closed form expression is also obtained for stationary stable Gaussian autoregressive sources with unit memory and per-letter $\mse$ distortion in \cite{gorbunov-pinsker1972a,derpich:2012} and for time-invariant stable or unstable Gaussian autoregressive sources with unit memory and average MSE distortion in \cite{stavrou2016}.}:
\begin{equation}
\label{eq:SRDscalar}
R^{\srd}(D)=\max \left\{0, \frac{1}{2}\log\left(A^2+\frac{\Sigma_{\bw}}{D}\right)\right\}.
\end{equation}
In \cite{tatikonda:2004}, it is also claimed that $R^{\srd}(D)$ for time-invariant $\mathbb{R}^p$-valued Gauss-Markov processes admits an explicit form
\begin{equation}
\label{eq:SRDvec}
R^{\srd}(D)=\frac{1}{2}\log\left| AA{\T}+\frac{p}{D}\Sigma_{\bw}\right|
\end{equation}
over the low distortion region of $D$ satisfying
\begin{equation}
\label{eqdcond}
\frac{D}{p} \leq \min_i \lambda_i \left(\frac{D}{p}AA{\T}+\Sigma_{\bw} \right)
\end{equation}
where $\lambda_i(\cdot)$ denotes the $i^{\th}$ eigenvalue. Based on a dynamic reverse-waterfilling algorithm, Stavrou {\it et. al.} in \cite{stavrou-charalambous-charalambous2016} constructed an iterative numerical algorithm to compute $R^{\srd}(D)$ for time-varying and time-invariant $\mathbb{R}^p$-valued Gauss-Markov processes, which extends \eqref{eq:SRDvec} to the entire positive region of $D$. 
Tanaka {\it et. al.} \cite{tanaka:2017}, on the other hand, derived the following semidefinite representation of $\widehat{R}^{\srd}(D)$ for all $D>0$:
\begin{align}
\widehat{R}^{\srd}(D)= \min_{P,Q\succ 0} & \; - \frac{1}{2} \log \det Q + \frac{1}{2} \log \det \Sigma_{\bw}. \label{eq:SRDsdp} \\
\text{s.t. } &  \;\; P \preceq APA{\T}+\Sigma_{\bw} \nonumber \\
& \;\; \trace(P) \leq D \nonumber \\
& \; \left[ \begin{array}{cc}
P-Q & PA{\T} \nonumber \\
AP & APA{\T}+\Sigma_{\bw}\end{array}\right]\succeq 0. \nonumber 
\end{align}
\par Unfortunately, the following simple numerical experiment shows that the results \eqref{eq:SRDscalar}, \eqref{eq:SRDvec}, and  \eqref{eq:SRDsdp} cannot be true simultaneously. Figure~\ref{fig:compare} shows $R^{\srd}(D)$ for an $\mathbb{R}^2$-valued Gauss-Markov process \eqref{prob_stat:eq.1} with $A=\left[\begin{array}{cc}6 & 0 \\ 0 & 1\end{array}\right]$ and $\Sigma_{\bw}=I_2$, plotted using the results of \cite{tatikonda:2004,stavrou-charalambous-charalambous2016,tanaka:2017}. The plot shows that \eqref{eq:SRDsdp} takes smaller values than \eqref{eq:SRDvec} and its extension to $D>0$ obtained in \cite{stavrou-charalambous-charalambous2016}. However, this is a contradiction to our earlier observation that ${R}^{\srd}(D)\leq\widehat{R}^{\srd}(D)$.

\section{Main Results}
\label{sec:main_results}
In this section, we establish the following statements.
\begin{itemize}
\item[(i)] The expression \eqref{eq:SRDvec} of $R^{\srd}(D)$ for $\mathbb{R}^p$-valued Gauss-Markov process is not correct, even in the region of $D$ satisfying \eqref{eqdcond}. Consequently, its extension derived in \cite{stavrou-charalambous-charalambous2016} does not compute the value of $R^{\srd}(D)$ correctly.
\item[(ii)] For $\mathbb{R}^p$-valued Gauss-Markov processes, it turns out that $\widehat{R}^{\srd}(D)={R}^{\srd}(D)$. Thus, \eqref{eq:SRDsdp} provides a semidefinite representation of $R^{\srd}(D)$.
\end{itemize}

We first show (i) by means of a simple counterexample.

\subsection{Counterexample}\label{subsec:counterexample}

\par In what follows, we show that if \eqref{eq:SRDscalar} holds then \eqref{eq:SRDvec} does not hold. To see this, consider an $\mathbb{R}^2$-valued process  \eqref{prob_stat:eq.1} with $A=\left[\begin{array}{cc}a & 0 \\ 0 & 0 \end{array}\right]$ and $\Sigma_{\bw}=I_2$.
If \eqref{eq:SRDvec} holds, we have
\begin{align}
R^{\srd}(D)&=\frac{1}{2}\log \left|AA{\T}+\frac{2}{D}\Sigma_{\bw} \right|, \nonumber \\
&=\frac{1}{2}\log \left(a^2+\frac{2}{D}\right)\frac{2}{D}. \label{exsrdf}
\end{align}
According to \eqref{eqdcond}, the above expression \eqref{exsrdf} is valid for all $D \leq 2$.
\begin{figure}[t]
    \centering
 \includegraphics[width=\columnwidth]{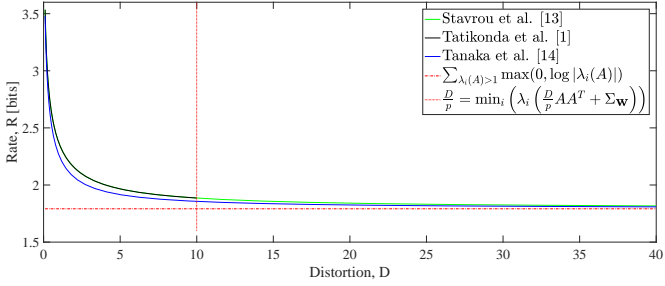}
    \caption{$R^{\srd}(D)$ plotted based on the methods of \cite{tatikonda:2004,stavrou-charalambous-charalambous2016,tanaka:2017} for an unstable time-invariant $\mathbb{R}^2$-valued Gauss-Markov source.}
    \label{fig:compare}
\end{figure}
On the other hand, notice that the considered $\mathbb{R}^2$-valued Gauss-Markov process can be viewed as two individual scalar Gauss-Markov processes:
\begin{align*}
{\bf x}_{1,t+1}&=a{\bf x}_{1,t}+{\bf w}_{1,t}, \;\; \bw_{1,t}\sim N(0,1), \\
{\bf x}_{2,t+1}&={\bf w}_{2,t}, \;\; {\bf w}_{2,t}\sim N(0,1).
\end{align*}
Applying \eqref{eq:SRDscalar} to each process, we have
\begin{align*}
R_1^{\srd}(D_1) &= \max\left\{0, \frac{1}{2}\log\left(a^2+\frac{1}{D_1}\right)\right\}, \\
R_2^{\srd}(D_2) &= \max\left\{0, \frac{1}{2}\log\left(\frac{1}{D_2}\right)\right\}.
\end{align*}
Notice that for all $D_1$ and $D_2$ satisfying $D_1+D_2=D$, we must have
\begin{equation}
\label{eqineq}
R^{\srd}(D) \leq R_1^{\srd}(D_1)+R_2^{\srd}(D_2).
\end{equation}
Now, if the expression \eqref{exsrdf} is correct, the left hand side ($\lhs$) of \eqref{eqineq} is
\begin{align}
R^{\srd}(1.5)&=\frac{1}{2}\log \left(a^2+\frac{4}{3}\right)\frac{4}{3}\nonumber \\
&=\frac{1}{2}\log \left(\frac{4}{3}a^2+\frac{16}{9}\right). \label{eqlhs}
\end{align}
The right hand side ($\rhs$) of \eqref{eqineq} is
\begin{align}
R_1^{\srd}(0.5)+R_2^{\srd}(1)
&=\frac{1}{2}\log \left(a^2+2\right)+0 \nonumber \\
&=\frac{1}{2}\log \left(a^2+2\right). \label{eqrhs}
\end{align}
(Notice that $D_2=1$ is achievable with zero-rate.)
However, \eqref{eqlhs}$>$\eqref{eqrhs} whenever $a^2> \frac{2}{3}$. This is a  contradiction to \eqref{eqineq}.

\begin{remark}
The above counterexample implies that there is a flaw in the dynamic reverse-waterfilling argument in \cite[Eg. (15)]{tatikonda:2004}. More precisely, unless the source process is $\IID$, (i.e., $A=0$) assigning equal distortions to each dimension is not optimal. Fig. \ref{fig:compare:white} shows that all the results coincide when $A=0$.
\end{remark}

\begin{figure}[t]
    \centering
 \includegraphics[width=\columnwidth]{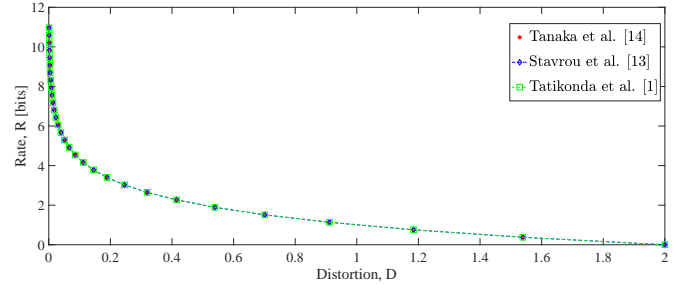}
    \caption{$R^{\srd}(D)$ plotted based on the methods of \cite{tatikonda:2004,stavrou-charalambous-charalambous2016,tanaka:2017}  for time-invariant $\mathbb{R}^2$-valued $\IID$ Gaussian sources, i.e., when $A=0$.}
    \label{fig:compare:white}
\end{figure}

\subsection{Semidefinite representation}

As the second main result of this paper, we show the statement (ii). 

\begin{theorem}
\label{theomain2}
For $\mathbb{R}^p$-valued Gauss-Markov processes, we have $R^{\srd}(D)=\widehat{R}^{\srd}(D)$.
\end{theorem}
\begin{proof}
See Appendix~\ref{appendix:finite_srdf}.
\end{proof}

Notice that while a semidefinite representation of $\widehat{R}^{\srd}(D)$ has been obtained in \cite{tanaka:2017}, no such expression is available for $R^{\srd}(D)$ in the literature. Hence, Theorem~\ref{theomain2} is a new result obtained in this paper for the first time. While we are not aware of an analytical expression of $R^{\srd}(D)$ for multidimensional Gauss-Markov processes, Theorem~\ref{theomain2} shows that $R^{\srd}(D)$ can be computed easily by semidefinite programming.
  
It is straightforward to verify that for scalar Gauss-Markov processes, the left hand side of \eqref{eq:SRDsdp} simplifies to \eqref{eq:SRDscalar}. This shows the correctness of \eqref{eq:SRDscalar}. Thus, the counterexample in the previous subsection implies that the formula \eqref{eq:SRDvec} reported in \cite{tatikonda:2004} is not correct.
%
%
%
%

\section{Conclusions}\label{sec:conclusions}

\par We revisited the problem of computing the asymptotic limit of SRD function for time-invariant vector-valued Gauss-Markov sources subject to a per-letter MSE distortion, introduced in \cite{tatikonda:2004}. We showed, via a counterexample, that the closed form expression of the SRD function derived in \cite[eq. (17)]{tatikonda:2004}  using the dynamic reverse-waterfilling algorithm suggested in \cite[eq. (15)]{tatikonda:2004} is not correct even in the low distortion region. We also showed that the the SRD function is semidefinite representable and thus it can be computed numerically.

%
%
%
%

\appendices

\section{Proof of Lemma \ref{lemma:2}}\label{proof:lemma:2}

Suppose the minimizer process $\by_t$ defined by \eqref{linear:expression} is given. Construct a new process $\tilde{\by}_t$ by
\begin{subequations}
\begin{align}
\bq_t&=\bar{A}_t\bx_t+\bv_t, \label{eq:ytilde1}\\
\tilde{\by}_t&=\mathbb{E}(\bx_t|\bq^t), \label{eq:ytilde2}
\end{align}
\end{subequations}
where $\bv_t$ is the same random process as in \eqref{linear:expression}. Notice that $\tilde{\by}_t$ can be written in a recursive form as
\begin{subequations}
\begin{align}
\tilde{\by}_t&=A\tilde{\by}_{t-1}+L_t(\bq_t-\bar{A}_tA\tilde{\by}_{t-1}) \label{eq:kf1}\\
&=L_t\bar{A}_t\bx_t+(I-L_t\bar{A}_t)A\tilde{\by}_{t-1}+L_t\bv_t, \label{eq:kf2}
\end{align}
\end{subequations}
where $L_t$, $t=0,\ldots,n$ are the Kalman gains.

It is sufficient to show that 
\begin{align}
I(\bx^n; \by^n)&=I(\bx^n; \tilde{\by}^n),\text{ and} \label{eqmi} \\
\mathbb{E}\|\bx_t-\by_t\|^2&\geq \mathbb{E}\|\bx_t-\tilde{\by}_t\|^2,~~\forall t=0,\ldots,n. \label{eqmse}
\end{align}
First, we show \eqref{eqmi}.

\underline{Proof of \eqref{eqmi}}: Notice that
\begin{subequations}
\label{eqmixyytilde}
\begin{align}
I(\bx^n;\by^n)&\stackrel{(a)}=\sum_{t=0}^n I(\bx^t;\by_t|\by^{t-1}) \stackrel{(b)}=\sum_{t=0}^n I(\bx_t;\by_t|\by^{t-1}), \\
I(\bx^n;\tilde{\by}^n)&\stackrel{(c)}=\sum_{t=0}^n I(\bx^t;\tilde{\by}_t|\tilde{\by}^{t-1}) \stackrel{(d)}=\sum_{t=0}^n I(\bx_t;\tilde{\by}_t|\tilde{\by}^{t-1}).
\end{align}
\end{subequations}
Equalities (a) and (c) follow from the problem formulation \eqref{mutual:information}, (b) follows from \eqref{linear:expression}, and (d) follows from \eqref{eq:kf2}. Hence, it is sufficient to show that 
\begin{equation}
I(\bx_t;\by_t|\by^{t-1})=I(\bx_t;\tilde{\by}_t|\tilde{\by}^{t-1}), \;\; \forall t=0,\ldots,n,
\label{eqmixy}
\end{equation}
holds. By \eqref{linear:expression} and \eqref{eq:ytilde1}, we have $\by_t=\bq_t+\bar{B}_t\by^{t-1}$. Thus, for all $t=0,\ldots, n$, $\by^t$ and $\bq_t$ are related by an invertible linear map
\begin{equation}
\left[\begin{array}{cccc}
I_p & 0  & \cdots & 0\\
* & I_p & \ddots & \vdots\\
\vdots & \ddots & \ddots &0 \\
* & \cdots & * & I_p 
\end{array}\right]
\left[\begin{array}{c}
\by_0 \\
\vdots \\
\by_t
\end{array}\right] =
\left[\begin{array}{c}
\bq_0 \\
\vdots \\
\bq_t
\end{array}\right]. 
\label{eqinvertible}
\end{equation}
Thus, we have
\begin{align}
I(\bx_t;\by_t|\by^{t-1})
&=I(\bx_t;\bq_t+\bar{B}_t\by^{t-1}|\by^{t-1}) \nonumber \\
&=I(\bx_t;\bq_t|\by^{t-1}) \nonumber \\
&\stackrel{(e)}=I(\bx_t;\bq_t|\bq^{t-1}), \label{eq23}
\end{align}
where (e) holds since $\by^{t-1}$ and $\bq^{t-1}$ are related by an invertible map \eqref{eqinvertible}.
Since $\tilde{\by}_t$ is the output of the Kalman filter, we have the following conditional independence:
\begin{equation}
\tilde{\by}^{t} \leftrightarrow   \bq^{t}\leftrightarrow \bx_t, \;\;\;
\bq^{t} \leftrightarrow \tilde{\by}^{t} \leftrightarrow \bx_t.
\end{equation}

The first relationship holds since $\tilde{\by}^{t}$ is a deterministic function of $\bq^{t}$. The second relationship holds because of the orthogonality principle $\mathbb{E}\bq^t (\bx_t-\tilde{\by}_t){\T}=0$ (which, together with the Gaussian property, implies that $\bq^t$ and $\bx_t-\tilde{\by}_t$ are independent) of the minimum $\mse$. Similarly, we have
\begin{equation}
\tilde{\by}^{t-1} \leftrightarrow \bq^{t-1} \leftrightarrow \bx_t,  \;\;\; \bq^{t-1} \leftrightarrow   \tilde{\by}^{t-1}\leftrightarrow \bx_t.
\end{equation}
Thus, by the data processing inequality, we have
\begin{equation}
I(\bx_t;\bq^t)=I(\bx_t;\tilde{\by}^t) \text{ and } I(\bx_t;\bq^{t-1})=I(\bx_t;\tilde{\by}^{t-1}).
\label{eq13}
\end{equation}
Therefore, 
\begin{subequations}
\label{eq14}
\begin{align}
I(\bx_t;\bq_t|\bq^{t-1})&=I(\bx_t;\bq^t)-I(\bx_t;\bq^{t-1}) \\
&=I(\bx_t;\tilde{\by}^t)-I(\bx_t;\tilde{\by}^{t-1}) \label{eq14b} \\
&= I(\bx_t;\tilde{\by}_t|\tilde{\by}^{t-1}).
\end{align}
\end{subequations}
Equality \eqref{eq13} is used in step \eqref{eq14b}.
From \eqref{eq23} and \eqref{eq14}, we obtain \eqref{eqmixy}.\\
Next, we prove \eqref{eqmse}.

\underline{Proof of \eqref{eqmse}}:  Denote by $\sigma(\bq^t)$ the $\sigma$-algebra generated by the $\rv$ $\bq^t$.
Observe that
$\tilde{\by}_t$ is $\sigma(\bq^t)$-measurable, since it is the output of the Kalman filter \eqref{eq:kf1}.
Since $\tilde{\by}_t$ is the least $\mse$ estimate of $\bx_t$ given $\bq^t$, $\tilde{\by}_t$ is the minimizer of $\mse$ $\mathbb{E}\|\bx_t-\by_t'\|^2$ in the class of all $\sigma(\bq^t)$-measurable functions $\by_t'$.
However, because of the invertible relationship \eqref{eqinvertible}, $\by_t$ is also a $\sigma(\bq^t)$-measurable function.
Therefore, $\by_t$ cannot attain a strictly smaller $\mse$ than $\tilde{\by}_t$. Thus, we obtain \eqref{eqmse}.\\
This completes the proof.\qed

\section{Proof of Lemma \ref{lemma:3}}\label{proof:lemma:3}

\par If $R_{0,n}^{\srd}(D)<\infty$, then, matrices $\bar{A}_t$ and $\Sigma_{\bv_t}$ in \eqref{eq:xqy} must satisfy 
\begin{equation}
\label{eqnullspace}
\text{Im}(\bar{A}_t) \subseteq \text{Im}(\Sigma_{\bv_t}), \; \forall t=0,\ldots,n.
\end{equation}
(Otherwise there exists a subspace component of $\bq_t$ that deterministically depends on $\bx_t$, implying \eqref{eq23} is unbounded.)
Let  
\[
\Sigma_{\bv_t}=\left[\begin{array}{cc}\bu_{1,t} & \bu_{2,t}\end{array}\right]
\left[\begin{array}{cc}\Sigma_{\bv_t} &0 \\
0 & 0 \end{array}\right]
\left[\begin{array}{c}\bu_{1,t}^{\T} \\ \bu_{2,t}^{\T} \end{array}\right],
\]
be the singular value decomposition such that ${\bu}_t\triangleq\left[\begin{array}{cc}\bu_{1,t} & \bu_{2,t}\end{array}\right]$ is an orthonormal matrix and $\Sigma_{\bv_t} \succ 0$. By \eqref{eqnullspace}, we have $\bu_{2,t}^{\T} \bar{A}_t=0$. Now, if we set $E_t\triangleq \bu_{1,t}^{\T} \bar{A}_t$, $\bz_t\triangleq \bu_{1,t}^{\T} \bv_t$, and ${\bf p}_t \triangleq E_t\bx_t+\bz_t$, it is easy to check that $\bz_t\sim N(0, \Sigma_{\bz_t})$ and
\begin{equation}
\label{equqp}
\bu_t^{\T} \bq_t=\left[\begin{array}{c}\bu_{1,t}^{\T} \\ \bu_{2,t}^{\T} \end{array}\right] \bq_t= \left[\begin{array}{c}{\bf p}_t \\ 0 \end{array}\right].
\end{equation}
Moreover, $\by_t$ defined by \eqref{eq:xqy} can be written as
\begin{align*}
\by_t&= \mathbb{E}(\bx_t|\bq^t) \\
&\stackrel{(a)}=\mathbb{E}(\bx_t|\bu_k^{\T}\bq_k: k=0,\ldots,t) \\
&\stackrel{(b)}= \mathbb{E}(\bx_t|{\bf p}^t),
\end{align*}
where equality $(a)$ holds because $\bu_t$ is invertible, and $(b)$ is due to \eqref{equqp}. This completes the proof. \qed

\section{Proof of Theorem~\ref{theomain2}}\label{appendix:finite_srdf}

\par Notice that the process $\by_t$ in Lemma \ref{lemma:3}, \eqref{eqpy} can be recursively computed by the Kalman filter
\begin{align*}
&\by_t=\by_{t|t-1}+P_{t|t-1}E_t\T(E_t P_{t|t-1}E_t\T+\Sigma_{\bz_t})^{-1}({\bf p}_t-E_t\by_{t|t-1}) \\
&\by_{t|t-1}= A\by_{t-1},
\end{align*}
where $P_{t|t-1}$ is the solution to the Riccati recursion
\begin{align*}
P_{t|t-1}&=AP_{t-1|t-1}A{\T}+\Sigma_{\bw} \\
P_{t|t}&=(P_{t|t-1}^{-1}+E_t\T \Sigma_{\bz_t}^{-1}E_t)^{-1}.
\end{align*}
Since $P_{t|t-1}$ and $P_{t|t}$ can be interpreted as $\mse$ covariance matrices, we have
\begin{subequations}
\label{eqmilog}
\begin{align}
&I(\bx_t;\bq_t|\bq^{t-1}) \nonumber\\
&\stackrel{(a)}=I(\bx_t;{\bf p}_t|{\bf p}^{t-1}) \\
&=h(\bx_t|{\bf p}^{t-1})-h(\bx_t|{\bf p}^t) \\
&=\frac{1}{2}\log\det (AP_{t-1|t-1}A{\T}+\Sigma_{\bw})-\frac{1}{2}\log\det P_{t|t},
\end{align}
\end{subequations}
where (a) is due to \eqref{equqp}. Combining \eqref{eqmixyytilde}, \eqref{eq23} and \eqref{eqmilog}, we have shown that $I(\bx^n;\by^n)$ can be written using variables $P_{t|t}$, $t=0,\ldots, n$. Since we can also write $\mathbb{E}\|\bx_t-\by_t\|^2=\trace({P}_{t|t})$, the finite-time horizon Gaussian $\srd$ problem \eqref{prob_stat1} can be written as a non-convex optimization problem in terms of variables $P_{t|t}, E_t, \Sigma_{\bz_t}$, $t=0,\ldots, n$. Nevertheless, by employing the variable elimination technique discussed of \cite[Section IV]{tanaka:2017}, we can show that \eqref{prob_stat1} is semidefinite representable as follows.
\begin{subequations}
\label{eqsrdtv}
\begin{align}
&R_{0,n}^{\srd}(D)= \min_{\substack{P_{t|t}\\t=0,\ldots,n}} \quad \tfrac{1}{n+1} \biggl[ 
\tfrac{1}{2}\log\det \Sigma_{\bx_0}-\tfrac{1}{2}\log\det P_{0|0} \biggr. \nonumber \\
&\biggl. +\sum_{t=1}^n \left( \tfrac{1}{2}\log\det (AP_{t-1|t-1}A^{\T} +\Sigma_{\bw})-\tfrac{1}{2}\log\det P_{t|t} \right)
\biggr].  \label{eqsrdtv1}\\
&\text{s.t.} \quad 0 \prec P_{t|t} \preceq AP_{t-1|t-1}A^{\T} +\Sigma_{\bw},~ t=1, \ldots, n \label{eqsrdtv2}\\
& \hspace{5ex} 0 \prec P_{0|0} \preceq \Sigma_{\bx_0} \label{eqsrdtv3}\\
& \hspace{5ex} \trace(P_{t|t})\leq D,~t=0,\ldots, n \label{eqsrdtv4}
\end{align}
\end{subequations}
This can be reformulated as a convex optimization problem in terms of $\{P_{t|t}, Q_t:~t=0,\ldots,n\}$:
\begin{align}
&R_{0,n}^{\srd}(D)=\min -\tfrac{1}{n+1}\left(\sum_{t=0}^n \frac{1}{2}\log\det Q_t + c\right), \label{eqsdptv}\\
&\text{s.t. } Q_t\succ 0, \trace(P_{t|t})\leq D ,~t=0,\ldots,n\nonumber  \\
& P_{t|t} \preceq AP_{t-1|t-1}A{\T} +\Sigma_{\bw},~ t=1, \ldots, n\nonumber\\
& P_{0|0}\preceq \Sigma_{\bx_0},~P_{n|n}=Q_n \nonumber \\
& \left[\begin{array}{cc}
P_{t|t}-Q_t & P_{t|t}A{\T} \\
AP_{t|t} & AP_{t|t}A{\T}+\Sigma_{\bw}
\end{array}\right] \succeq 0,~~t=0,\ldots,n-1\nonumber 
\end{align}
Here, $c$ is a constant given by
\[
c=\frac{1}{2}\log\det \Sigma_{\bx_0}+\frac{n}{2}\log\det \Sigma_{\bw}.
\]
Next, we make a few observations regarding \eqref{eqsdptv}. First, \eqref{eqsdptv} is in the form of determinant-maximization problem \cite{vandenberghe1998determinant}. Therefore, standard $\sdp$ solvers can be used to solve it numerically. Second, once the optimal solution $P_{t|t}$, $t=0,\ldots, n$ of \eqref{eqsdptv} is found, the minimizer process 
\eqref{eqpy} for the Gaussian $\srd$ problem can be constructed by arbitrarily choosing matrices $E_t$ and $\Sigma_{\bz_t}\succ 0$ satisfying
\begin{equation}
\label{eqesigmav}
E_t\T \Sigma_{\bz_t}^{-1}E_t=P_{t|t}^{-1}-(AP_{t-1|t-1}A{\T}+\Sigma_{\bw})^{-1}.
\end{equation}
Since the rank of the $\rhs$ of \eqref{eqesigmav} can be different for each $t$, the size of the matrix $\Sigma_{\bz_t}$ is also different for each $t$. Thus, the dimension of the random vector ${\bf p}_t$ in \eqref{eqpy} is in general time-varying.

\par Next, we prove the following lemma.

\begin{lemma}{\ \\}\label{lemma:4}
There exist non-negative sequences $\{\epsilon_n\}$ and $\{\delta_n\}$ such that $\epsilon_n \searrow 0$ and $\delta_n \searrow 0$ as $n\rightarrow \infty$, and
\begin{align}
R_{0,n}^{\srd}(D) \geq f(D; \epsilon_n, \delta_n),\label{useful_inequality}
\end{align}
where
\begin{subequations}
\label{eqdeff}
\begin{align}
& f(D;\epsilon_n, \delta_n) \triangleq  \nonumber \\
& \min_P \quad \frac{1}{2}\log\det(APA{\T} + \Sigma_{\bw})-\frac{1}{2}\log\det P -\epsilon_n. \label{eqdeff1}\\
& \text{s.t.} \quad 0\prec P \preceq APA{\T}+\Sigma_{\bw}+\delta_n I_p \label{eqdeff2} \\
&\hspace{5ex} \trace(P)\leq D \label{eqdeff3}
\end{align}
\end{subequations}
\end{lemma}
\begin{proof}
\par Let $\{P_{t|t}\}_{t=0}^n$ be the minimizer sequence for \eqref{eqsrdtv}, and define $P\triangleq \frac{1}{n+1}\sum_{t=0}^n P_{t|t}$.
We first show that there exists a sequence $\delta_n \searrow 0$ such that \eqref{eqdeff2} holds for each $n$. From \eqref{eqsrdtv2}, we have 
\[ P \preceq APA{\T} +\Sigma_{\bw}+\frac{1}{n+1} P_{0|0}.\]
Thus, \eqref{eqdeff2} is feasible with a choice $\delta_n=\sigma_{\text{max}}(\frac{1}{n+1}P_{0|0})$ (the maximum singular value of $\frac{1}{n+1}P_{0|0}$).

\par Next, we show that there exists a sequence $\epsilon_n \searrow 0$ such that for each $n$ the objective function  \eqref{eqdeff1} is a lower bound of the objective function \eqref{eqsrdtv1}. Notice that \eqref{eqsrdtv1} without the minimization can be written as follows.
\begin{subequations}
\label{eqregroup}
\begin{align}
&\text{\eqref{eqsrdtv1}}=\tfrac{1}{n+1}\biggl[\tfrac{1}{2}\log\det \Sigma_{\bx_0}-\tfrac{1}{2}\log\det P_{0|0} \biggr. \nonumber\\
&+\biggl. \sum_{t=1}^n\left(\tfrac{1}{2}\log\det(AP_{t-1|t-1}A{\T}+\Sigma_{\bw})-\tfrac{1}{2}\log\det P_{t|t} \right) \biggr] \nonumber\\
&=\tfrac{1}{n+1}\biggl[\tfrac{1}{2}\log\det \Sigma_{\bx_0}-\tfrac{1}{2}\log\det(AP_{n|n}A{\T}+\Sigma_{\bw}) \biggr. \nonumber\\
&+\biggl. \sum_{t=0}^n\left(\tfrac{1}{2}\log\det(AP_{t|t}A{\T}+\Sigma_{\bw})-\tfrac{1}{2}\log\det P_{t|t} \right) \biggr] \nonumber\\
&=\tfrac{1}{n+1}\biggl(\tfrac{1}{2}\log\det \Sigma_{\bx_0}-\tfrac{1}{2}\log\det(AP_{n|n}A{\T}+\Sigma_{\bw}) \biggr) \label{eqregroup1}\\
&+\tfrac{1}{n+1} \sum_{t=0}^n\left(\tfrac{1}{2}\log\det \Sigma_{\bw}-\tfrac{1}{2}\log\det (P_{t|t}^{-1}+A{\T} \Sigma_{\bw}^{-1}A) \right). \label{eqregroup2}
\end{align}
\end{subequations}
Using the identity $\det X \leq \trace((X)/p)^p)$ for general $X\in \mathbb{S}_{++}^p$,
\begin{align*}
\det({AP_{n|n}A{\T}+\Sigma_{\bw}})&\leq \left(\frac{\trace(AP_{n|n}A{\T})+\trace(\Sigma_{\bw})}{p}\right)^p \\
&\leq \left(\frac{\sigma_{\text{max}}(AA{\T})D+\trace(\Sigma_{\bw})}{p}\right)^p.\end{align*}
Hence, there exists a positive constant $\gamma$ such that 
\begin{align}
\text{\eqref{eqregroup1}}\geq &\frac{1}{n+1}\Bigg(\frac{1}{2}\log\det \Sigma_{\bx_0}\nonumber\\
&-\frac{1}{2}\log\left(\frac{\sigma_{\text{max}}(AA{\T})D+\trace(\Sigma_{\bw})}{p}\right)^p
\Bigg) \nonumber \\
&\geq -\frac{1}{n+1}\gamma \nonumber \\
&=-\epsilon_n. \label{eq10}
\end{align}
 \par In the last line, we defined $\epsilon_n\triangleq \frac{1}{n+1}\gamma$. Moreover, \eqref{eqregroup2} is lower bounded as follows:
\begin{align}
\text{\eqref{eqregroup2}}&\stackrel{(a)}\geq \frac{1}{2}\log\det \Sigma_{\bw}+\frac{1}{2}\log\det(P^{-1}+A{\T} \Sigma_{\bw}^{-1}A) \nonumber \\
&=\frac{1}{2}\log\det (APA{\T}+\Sigma_{\bw})-\frac{1}{2}\log\det P,
\label{eq11}
\end{align}
where (a) follows from the fact that $\log\det({P^{-1}+A{\T} \Sigma_{\bw}^{-1}A})$ is convex in $P$, and Jensen's inequality \cite[Theorem 2.6.2]{cover-thomas2006}. 
Moreover, from \eqref{eq10} and \eqref{eq11}, we have 
\begin{align}
\text{\eqref{eqsrdtv1}}\geq \frac{1}{2}\log\det(APA{\T}+\Sigma_{\bw})-\frac{1}{2}\log\det P-\epsilon_n,\label{inequality}
\end{align}
which gives the desired inequality. This completes the proof.
\end{proof}
\par Suppose the conditions of Lemma \ref{lemma:4} hold. Then, by taking the limit in both sides of \eqref{useful_inequality} we obtain $R^{\srd}(D)\geq\lim_{n\longrightarrow\infty}{f}(D;,\epsilon_n,\delta_n)$. However, $\lim_{n\longrightarrow\infty}{f}(D;,\epsilon_n,\delta_n)=\widehat{R}^{\srd}(D)$. This implies that $\widehat{R}^{\srd}(D)\leq{R}^{\srd}(D)$.
\par The converse inequality, i.e., $\widehat{R}^{\srd}(D)\geq{R}^{\srd}(D)$,  holds in general, however, it can be shown following the steps of \cite[Section IV]{tanaka2015semidefinite}. Hence, we omit it. This completes the proof. \qed

%
%
%
%
\section*{Acknowledgement}
\label{sec:acknowledgement}

The authors would like to thank Prof. M. Skoglund, Prof. J. $\O$stergaard and Prof. C. D. Charalambous for fruitful discussions.

\bibliographystyle{IEEEtran}
\bibliography{string,references}

%
%
%
%
\end{document}